\documentclass{amsart}
\usepackage[ansinew]{inputenc}
\usepackage[english]{babel}
\usepackage{amsmath,latexsym,amssymb,verbatim}

\usepackage[all]{xy}

\usepackage{hyperref}

\newtheorem{theorem}{Theorem}[section]
\newtheorem{lemma}[theorem]{Lemma}
\newtheorem{proposition}[theorem]{Proposition}
\newtheorem{corollary}[theorem]{Corollary}

\newtheorem{remark}[theorem]{Remark}
\newtheorem{note}[theorem]{Note}

\newtheorem{Formula of adjoint functors}[theorem]{Formula of adjoint functors}

\newtheorem{example}[theorem]{Example}

\newtheorem{definition}[theorem]{Definition}
\newtheorem{notation}[theorem]{Notation}

\newtheorem{Adjunction formula}[theorem]{\indent\sc Adjunction formula}

\DeclareMathOperator{\limi}{{lim}}
\newcommand{\ilim}[1]{\,\underset{#1}{\underset{\to}{\limi}}\,}
\newcommand{\plim}[1]{\,\underset{#1}{\underset{\leftarrow}{\limi}}\,}

\DeclareMathOperator{\Hom}{{Hom}}
\DeclareMathOperator{\Ext}{{Ext}}

\DeclareMathOperator{\Coker}{{Coker}}
\DeclareMathOperator{\Ker}{{Ker}}
\DeclareMathOperator{\Ima}{{Im}}
\DeclareMathOperator{\RR}{{\mathcal R}}
\DeclareMathOperator{\ZZ}{{\mathbb Z}}
\DeclareMathOperator{\QQ}{{\mathbb Q}}

\NoCompileMatrices

\input arrow.tex

\begin{document}

\title{Mittag-Leffler  functors of modules}

\author{Carlos Sancho, Fernando Sancho and Pedro Sancho}
%
%
%


\begin{abstract} Finite modules, finitely presented modules and Mittag-Leffler modules are characterized by their behaviour by tensoring with direct products
of modules. In this paper, we study and characterize the functors of modules
that preserve direct products.
\end{abstract}

\maketitle

\section{Introduction}

Let $R\,$ be an associative ring with unit, we will say that $\mathbb{M}\,$ is an $\mathcal{R}$-\textit{module} (right $\mathcal{R}$-module) if $\mathbb{M}\,$ is a covariant additive functor from the category of $R$-modules (respectively, right $R$-modules) to the category of abelian groups.



Any right $R$-module $M\,$ produces an $\mathcal{R}$-module. Namely, the \textit{quasi-coherent $\RR$-module} $\mathcal M$ associated with a right $R$-module $M$ is defined by
$$\mathcal M(S)=M\otimes_R S,$$
for any $R$-module $S$. It is significant to note that the category of right $R$-modules
is equivalent to the category of quasi-coherent $\RR$-modules. Therefore, we can study modules through their functorial incarnation.


On the other hand, given an $\mathcal R$-module $\mathbb M$,  $\mathbb M^*$ is the  right $\mathcal R$-module defined as follows:
$$\mathbb M^*(N):=\Hom_{\RR}(\mathbb M,\mathcal N),$$
for any right $R$-module $N$. 

A relevant fact is that quasi-coherent modules are reflexive, that is,  the canonical morphism of $\RR$-modules $\mathcal M\to \mathcal M^{**}$ is an isomorphism  (\cite{Pedro4}).

Quasi-coherent modules preserve direct limits, that is: 

$$\mathcal M(\ilim{i\in I} N_i )=M\otimes_R \ilim{i\in I} N_i = \ilim{i\in I} (M\otimes_R N_i)=\ilim{i\in I} \mathcal M(N_i)$$ for any direct system of $R$-modules $\{N_i\}_{i\in I}$, where $I$ is  an upward directed set. 
Watts (\cite[Th 1.]{Watts}) proved that an $\RR$-module is quasi-coherent iff it is a right exact functor and preserves direct limits. An $\mathcal R$-module $\mathbb M$ preserves direct limits  iff there exists an exact sequence of morphisms of $\RR$-modules $$\oplus_{i\in I} \mathcal P_i^*\to \oplus_{j\in J} \mathcal Q_j^*\to\mathbb M\to 0,$$ where $P_i$ and $Q_j$ are finitely presented $R$-modules, for every $i,j$. Besides, this exact sequence is a projective presentation of $\mathbb M$ (Thm \ref{4.85}). Let $\langle\text{\sl Qs-ch}\rangle$ be the category of $\mathcal R$-modules that preserve direct limits. $\langle\text{\sl Qs-ch}\rangle$ is the smallest full subcategory of the category of $\RR$-modules  containing quassi-coherent modules stable by kernels, cokernels and direct limits. It can be proved that the category $\langle\text{\sl Qs-ch}\rangle$  is  equivalent to the category of functors from the category of finitely presented $R$-modules to the category of abelian groups.
In case  that $R$ is a field, an $\mathcal R$-module  preserves direct limits  iff  it is quasi-coherent.

The aim of this paper is to extend the notions of finite, finitely presented and Mittag-Leffler modules to $\langle\text{\sl Qs-ch}\rangle$ and give different characterizations of these functors.

Any $R$-module is a direct limit of finitely presented $R$-modules and 
it is well known that an $R$-module $M$  is  finitely presented iff $\mathcal M$ preserves direct products.

\begin{definition} We will say that $\mathbb M\in \langle\text{\sl Qs-ch}\rangle$ is an FP module if it preserves   direct products.
\end{definition}

Every  $\RR$-module $\mathbb M\in \langle\text{\sl Qs-ch}\rangle$ is a direct limit of FP modules (\ref{T5.17}). FP modules are characterized as follows.

\begin{theorem}  $\mathbb F\in \langle\text{\sl Qs-ch}\rangle$  is an FP module iff any of the following statements holds

 \begin{enumerate}

\item $\Hom_{\mathcal R}(\mathbb F,\ilim{i\in I} \mathbb M_i)=\ilim{i\in I}
\Hom_{\mathcal R}(\mathbb F,\mathbb M_i)$, for any direct system of $\mathcal R$-modules $\{\mathbb M_i\}_{i\in I}$.

\item $\mathbb F$ is reflexive and $\mathbb F^*\in \langle\text{\sl Qs-ch}\rangle$.

\item  There exists an exact sequence of $\RR$-modules $$\mathcal P_1^*\to \mathcal P_2^*\to \mathbb F\to 0,$$ where $P_1$ and $P_2$ are finitely presented $R$-modules.

\item There exists an exact sequence of $\RR$-modules
$$0\to \mathbb F \to \mathcal Q_1\to \mathcal Q_2,$$ where  $Q_1$ and $Q_2$ are finitely presented right $R$-modules.

\end{enumerate}
\end{theorem}

The category of FP modules is an abelian category. However,  in general, the category of finitely presented $R$-modules is not abelian. If $\mathbb F$ is an FP module, then
$\mathbb F^*$ is an FP module. However,  in general, if $P$ is a finitely presented $R$-module, then $P^*$ is not a finitely presented $R$-module.

\begin{definition} We will say that an $\RR$-module $\mathbb M\in \langle\text{\sl Qs-ch}\rangle$ is an ML module if the natural morphism $\mathbb M(\prod_i S_i)\to \prod_i \mathbb M(S_i)$ is injective for any set $\{S_i\}$ of  $R$-modules.
\end{definition}

A right $R$-module $M$ is a Mittag-Leffler module iff $\mathcal M$ is an ML
module. ML modules are chatacterized as follows.

\begin{theorem} Let $\mathbb M\in \langle\text{\sl Qs-ch}\rangle$. The following statements are equivalent:\begin{enumerate}

\item $\mathbb M$ is an ML module.

\item  $\mathbb M$ is a direct limit of FP  submodules.

\item  The kernel of every morphism
$\mathbb F\to\mathbb M$  is an FP module, for any FP module $\mathbb F$.

\end{enumerate}

\end{theorem}

If $M$ is a Mittag-Leffler module it is not true, as  a general rule, that $M$ is a direct limit of finitely presented submodules, nor is it true that the image of a morphism of $R$-modules $P\to M$ is a finitely presented module (where $P$ is a finitely presented $R$-module).

\begin{definition} Let $\mathbb M$ be an ML $\RR$-module.  $\mathbb M$  is said to be an SML module if for any FP submodule $\mathbb F\subseteq \mathbb M$ the dual morphism $\mathbb M^*\to \mathbb F$ is an epimorphism.\end{definition}

$M$ is a right strict Mittag-Leffler $R$-module  iff $\mathcal M$ is an SML module (\ref{7.8}). SML modules are chatacterized as follows.

\begin{theorem}  \label{smlFP} Let $\mathbb M\in \langle\text{\sl Qs-ch}\rangle$. The following statements are equivalent:\begin{enumerate}

\item $\mathbb M$ is an SML module.

\item $\mathbb M$ is a direct limit of  FP submodules $\mathbb F_i$, and the
morphism $\mathbb M^*\to  \mathbb F_i^*$ is an epimorphism, for any $i$.

\item There exists a monomorphism $\mathbb M\hookrightarrow \prod_{i\in I} \mathcal P_i$, where $P_i$ is a finitely presented (right) module, for each $i\in I$.   

\end{enumerate}

\end{theorem}

In particular, if $M$ is a strict Mittag-Leffler $R$-module, then it is a pure submodule of a direct product of finitely presented $R$-modules (this result can be found in  \cite{Guil}).

Finally we prove the following theorem.

\begin{theorem} Let $M$ be an $R$-module. Then,

\begin{enumerate}

\item $M$ is a Mittag-Leffler module iff the kernel of any morphism $\prod_{\mathbb N} \mathcal R\to \mathcal M$  preserves direct products.

\item $M$ is a strict Mittag-Leffler module iff the cokernel of any morphism $\mathcal M^*\to\oplus_{\mathbb N} \mathcal R$  is isomorphic  to an $\mathcal R$-submodule of a quasi-coherent module.
\end{enumerate}

\end{theorem}

This paper is self contained. Functorial characterizations of flat Mittag-Leffler modules and flat strict Mittag-Leffler modules are given in \cite{Pedro4}.

\section{Preliminaries}

\begin{remark} For the rest of the paper, every definition or statement is given with one module structure (left or right) on each of the modules appearing in that definition or statement;  we leave to the reader to do the respective definition or statement by interchanging  the left and right structures.
\end{remark}

\begin{notation} \label{nota2.1} Let $\mathbb M$ be a functor of $\RR$-modules.
For simplicity,
we will  sometimes use $m\in \mathbb M$  to denote $m\in \mathbb M(S)$. Given $m \in \mathbb M(S)$ and a morphism of  $R$-modules  $S \to S'$, we will  often denote by $m$ its image by the morphism $\mathbb M(S) \to \mathbb M(S')$.\end{notation}

\begin{remark} Direct limits, inverse limits of $\mathcal R$-modules and  kernels, cokernels, images, etc.,  of morphisms of $\mathcal R$-modules are regarded in the category of $\mathcal R$-modules. Besides,

$$\aligned   &  (\ilim{i\in I} \mathbb M_i)({S})=\ilim{i\in I} (\mathbb M_i({S})),\,
(\plim{j\in J} \mathbb M_j)({S})=\plim{j\in J} (\mathbb M_j({S})),\\&
(\Ker f)({S})=\Ker f_{S},\, (\Coker f)({S})=\Coker f_{S},\, (\Ima f)({S})=\Ima f_{S},
\endaligned $$
(where $I$ is an upward directed set and $J$ a downward directed set).
\end{remark}

 We will  denote by $\Hom_{\mathcal R}(\mathbb M,\mathbb M')$ the  family of all  morphisms of $\mathcal R$-modules from $\mathbb M$ to $\mathbb M'$.

\begin{proposition} \cite[2.11]{Pedro4} \label{trivial} Let $\mathbb M$ be a (left) $\RR$-module and let $\mathbb N$ be a (right) $\mathcal R$-module. Then,
$$\Hom_{\mathcal R}(\mathbb M,\mathbb N^*)=\Hom_{\mathcal R}(\mathbb N,\mathbb M^*),\, f\mapsto \tilde f,$$
where $\tilde f$ is defined as follows: $\tilde f(n)(m):=f(m)(n)$, for any $m\in\mathbb M$ and $n\in\mathbb N$.

\end{proposition}


\begin{definition} Let $\mathbb M$ be  an $\mathcal R$-module. We will  say that
$\mathbb M^*$ is the dual (right) $\RR$-module of $\mathbb M$.
We will  say that  an $\mathcal R$-module  ${\mathbb M}$ is reflexive if the natural morphism $${\mathbb M}\to {\mathbb M}^{**}, m\mapsto \tilde m \text{ (for any }m\in \mathbb M(S))$$ 
 is an isomorphism, where $\tilde m_{S'}(w):=w_S(m)$ (for any $w\in \mathbb M^*(S')=\Hom_{\RR}(\mathbb M,\mathcal S')$).\end{definition}

\begin{proposition} \label{trivialon} Let $\mathbb M$ and $\mathbb M'$ be reflexive functors of $\RR$-modules, $f\colon \mathbb M\to \mathbb M'$ a morphism of $\RR$-modules and $f^*\colon \mathbb N^*\to \mathbb M^*$ the dual morphism.
Then, $\Ker f =(\Coker f^*)^*$.

\end{proposition}

\subsection{Quasi-coherent modules}

\begin{definition} Let $M$ (resp. $N$,  $V$, etc.) be a right $R$-module. We will  denote by  ${\mathcal M}$  (resp. $\mathcal N$, $\mathcal V$, etc.)   the $\mathcal R$-module defined by ${\mathcal M}({S}) := M \otimes_R {S}$ (resp. $\mathcal N({S}):=N\otimes_R {S}$,  $\mathcal V({S}):=V\otimes_R {S}$, etc.). $\mathcal M$  will be called the quasi-coherent $\mathcal R$-module associated with $M$.
\end{definition}

\begin{proposition}  \cite[2.4]{Pedro4} The functors
$$\aligned \text{Category of right $R$-modules } & \to \text{ Category of quasi-coherent $\mathcal R$-modules }\\ M & \mapsto \mathcal M\\ \mathcal M(R) & \leftarrow\!\shortmid \mathcal M\endaligned$$
stablish an equivalence  of categories. In particular,
$${\rm Hom}_{\mathcal R} ({\mathcal M},{\mathcal M'}) = {\rm Hom}_R (M,M').$$
\end{proposition}

%
%

For another, slightly different, version of this proposition see \cite[1.12]{Amel}.

Let $f_R\colon M\to N$ be a morphism of $R$-modules and $f\colon \mathcal M\to \mathcal N$
the associated morphism of $\mathcal R$-modules. Let $C=\Coker f_R$, then $\Coker f=\mathcal C$, which is a quasi-coherent module.

Let $\mathbb M$ be an $\RR$-module. Observe that $\mathbb M(R)$ is naturally a right $R$-module: Given $r\in R$, consider the morphism of $R$-modules
$\cdot r\colon R\to R$. Then, 
$$m\cdot r:=\mathbb M(\cdot r)(m), \text{ for any } m\in \mathbb M(R).$$

\begin{proposition} \label{tercer}
For every  ${\mathcal R}$-module $\mathbb M$ and every right $R$-module $M$, it is satisfied that
$${\rm Hom}_{\mathcal R} ({\mathcal M}, \mathbb M) = {\rm Hom}_R (M, \mathbb M(R)),\, f\mapsto f_R.$$
\end{proposition}

\begin{notation} Let $\mathbb M$ be an $\mathcal R$-module.
We will  denote by $\mathbb M_{qc}$ the quasi-coherent module associated with
the $R$-module $\mathbb M(R)$, that is, $$\mathbb M_{qc}({S}):=\mathbb M(R)\otimes_R S.$$
Given $s\in S$, consider the morphism of $R$-modules $R\overset{\cdot s}\to  S$, $r\mapsto r\cdot s$. Then, we have the morphism $\mathbb M(\cdot s)\colon
\mathbb M(R)\to \mathbb M(S)$, $m\mapsto \mathbb M(\cdot s)(m)=:m\cdot s$.
\end{notation}

\begin{proposition} \cite[2.7]{Pedro4} \label{tercerb} For each  $\mathcal R$-module $\mathbb M$ one has the natural morphism $$\mathbb M_{qc}\to \mathbb M, \,m\otimes s\mapsto m\cdot s,$$ for any $m\otimes s\in \mathbb M_{qc}({S})= \mathbb M(R)\otimes_R S$, and a functorial equality
$$\Hom_{\mathcal R}(\mathcal N,\mathbb M_{qc})=\Hom_{\mathcal R}(\mathcal N,\mathbb M),$$
for any quasi-coherent $\mathcal R$-module $\mathcal N$.
\end{proposition}

%
%

Obviously, an $\mathcal R$-module $\mathbb M$ is a quasi-coherent module iff
the natural morphism $\mathbb M_{qc}\to \mathbb M$ is an isomorphism.

\begin{theorem} \cite[2.14]{Pedro4}  \label{prop4}
Let $M$ a right $R$-module and let $M'$ be an $R$-module. Then, $${M} \otimes_{R} {M'}={\Hom}_{\mathcal R} ({\mathcal M^*}, {\mathcal M'}),\, m\otimes m'\mapsto \tilde{m\otimes m'},$$
where $ \tilde{m\otimes m'}(w):=w(m)\otimes m'$, for any $w\in \mathcal M^*$.
\end{theorem}

%

\begin{note} \label{2.12N}  It is easy to prove that the morphism $$f=\sum_{i=1}^n m_i\otimes m'_i\in {\Hom}_{\mathcal R} ({\mathcal M^*}, {\mathcal M'})={M} \otimes_{R} { M'},$$ is equal to the composite morphism
$\mathcal M^*\overset g\to \mathcal L\overset h\to \mathcal M'$, where $L$ is the free module of basis $\{l_1,\ldots,l_n\}$, $g:=\sum_{i} m_i\otimes l_i\in
{\Hom}_{\mathcal R} ({\mathcal M^*}, {\mathcal L})={M} \otimes_{R} { L}$
and $h(l_i):=m'_i$ for any $i$. Observe that $\Ima f\subseteq \Ima h$.

\end{note}

%
%

If we make $\mathcal M'=\mathcal R$ in the previous theorem, we obtain the following theorem.

\begin{theorem} \cite[2.16]{Pedro4}  \label{reflex}
Let $M$ be a right $R$-module. Then, the canonical morphism $${\mathcal M}\to {\mathcal M^{**}},$$ is an isomorphism. That is,  quasi-coherent modules are reflexive.
\end{theorem}


\begin{definition} Let $M$ be an $R$-module. $\mathcal M^*$  will be called the $\mathcal R$-module scheme associated with $M$.
\end{definition}

\begin{theorem} \cite[2.10]{Pedro4}  \label{L5.11} Let $\{\mathbb M_i\}$ be a direct system of $\mathcal R$-modules. Then,
$$\Hom_{\mathcal R}(\mathcal N^*,\ilim{i} \mathbb M_i)=\ilim{i}\Hom_{\mathcal R}(\mathcal N^*, \mathbb M_i),$$
for any $R$-module $N$.
\end{theorem}

%

\subsection{Dual module of a direct product of $\mathcal R$-modules}

\begin{proposition} \label{reFPro2} Let $\{\mathbb M_i\}_{i\in I}$ be a set of  $\mathcal R$-modules and let $\mathbb N$ be an $\mathcal R$-module that preserves direct sums. Then, the natural  morphism
$$\oplus_{i\in I} \Hom_{\mathcal R}(\mathbb  M_i,\mathbb N)\to \Hom_{\mathcal R}(\prod_{i\in I} \mathbb  M_i,\mathbb N), (f_i)_{i\in I}\mapsto \sum_{i\in I} f_i,
$$ is an isomorphism, where  $(\sum_{i\in I} f_i)(m_i):=\sum_{i\in I} f_i(m_i)$ for any $(m_i)\in \prod_{i\in I} \mathbb M_i$.

\end{proposition}

\begin{proof}
The morphism $(f_i)_{i\in I}\mapsto \sum_{i\in I} f_i$
is obviously injective.

For any $i'\in I$, we have the obvious inclusion morphism $\mathbb M_{i'}\subseteq \prod_{i\in I} \mathbb M_i$.

Given $f\in \Hom_{\mathcal R}(\prod_{i\in I} \mathbb  M_i,\mathbb N)$, put $J:=\{i\in I\colon f_i:=f_{|\mathbb  M_i}\neq 0\}$. For each $j\in J$, there exist   an $R$-module $S_j$  and $m_j\in\mathbb M_j(S_j)$  such that $0\neq {f_j}_{S_j}(m_j)\in N\otimes_R S_j$. $\Hom_{\RR}(\mathcal S^*_j,\mathbb M_j)=\mathbb M_j(S_j)$, by the Yoneda Lemma.  Hence, we have the morphism $g_j\colon \mathcal S^*_j\to \mathbb M_j$ defined by
$m_j$. Consider the morphism
$$g\colon \prod_{j\in J} \mathcal S_j^*\to \prod_{j\in J}\mathbb M_j,\, g((w_j)_{j\in J}):=(g_j(w_j))_{j\in J}.$$ 
Put $h:=f\circ g$.   
On one hand  $h_{|\mathcal S_j^*}(Id_j)= {f_j}_{S_j}(m_j)\neq 0$, for any $j\in J$, where $Id_j\in \mathcal S^*_j(S_j)=\Hom_{R}(S_j,S_j)$ is the identity morphism. On the other hand
$$\aligned \Hom_{\RR}( \prod_{j\in J} \mathcal S_j^*,\mathbb N) & =
\Hom_{\RR}( (\oplus_{j\in J} \mathcal S_j)^*,\mathbb N)=\mathbb N(\oplus_{j\in J}  S_j)\\ & =\oplus_{j\in J}  \mathbb N(S_j ) =\oplus_{j\in J}\Hom_{\RR}(  \mathcal S_j^*,\mathbb N).\endaligned$$
Hence, $h=\sum_{j\in J} h_{|\mathcal S^*_j}$, where $ h_{|\mathcal S^*_j}=0$, for  all $j\in J$ except for a finite number of them.
Then,  $\# J<\infty$. 

Finally, let us prove that $f=\sum_{j\in J} f_j$: Let $m=(m_i)\in \prod_{i\in I} \mathbb M_i(S)$, Consider  

\end{proof}

\begin{corollary} \label{reFPro} Let $\{\mathbb M_i\}_{i\in I}$ be a set of  $\mathcal R$-modules. Then, $(\prod_{i\in I}\mathbb  M_i)^*=\oplus_{i\in I}\mathbb  M_i^*$.
\end{corollary}

\begin{corollary} \label{2.27} Let $\{\mathbb M_i\}_{i\in I}$ be a set of reflexive $\RR$-modules. Then, $\oplus_i\mathbb M_i$ and  $\prod_i\mathbb M_i$ are reflexive $\RR$-modules. \end{corollary}

%
%
%
%
%
%
%
%
%
%

\section{Functors that preserve direct limits}

\begin{definition} Let $\mathbb M$ be an $\RR$-module. We will  say that $\mathbb M$
preserves direct limits if  $\mathbb M(\ilim{i}S_i)=\ilim{i}\mathbb M( S_i)$
for every direct system of  $R$-modules  $\{S_i\}_{i\in I}$.
\end{definition}

\begin{example} Quasi-coherent modules preserve direct limits.

\end{example}

\begin{proposition} \label{1.9} Let $P$ be a finitely presented (right) module and $\{\mathbb M_i\}$ a direct system of $\mathcal R$-modules. Then,
$$\Hom_{\mathcal R}(\mathcal P,\ilim{i} \mathbb M_i)=\ilim{i}
\Hom_{\mathcal R}(\mathcal P,\mathbb M_i).$$
In particular, $\mathcal P^*$ preserves direct limits.\end{proposition}

\begin{proof} By \ref{tercer},
$\Hom_{\mathcal R}(\mathcal P,\ilim{i} \mathbb M_i)=
\Hom_{R}( P,\ilim{i} \mathbb M_i(R))=\ilim{i} \Hom_{R}( P,\mathbb M_i(R)) $ $=$ $\ilim{i}
\Hom_{\mathcal R}(\mathcal P,\mathbb M_i)$.
\end{proof}

\begin{proposition} \label{4.4}  An $R$-module $M$ is finitely presented  iff $\mathcal M^*$ preserves direct limits.

\end{proposition}

\begin{proof} $\Leftarrow)$ Any $R$-module is a direct limit of
finitely presented modules. Write $M=\ilim{i} P_i$, where $P_i$ is a finitely presented module, for any $i$. Observe that
$$Id\in \Hom_{\mathcal R}(\mathcal M,\mathcal M)={\mathcal M^*}(M)=\ilim{i} {\mathcal M^*}(P_i)=\ilim{i} \Hom_{\mathcal R}(\mathcal M,\mathcal P_i),$$ hence $Id$ factors through a morphism $f_i\colon M\to P_i$,  for some $i$. Then, $M$ is a direct summand of $P_i$, and it is finitely presented.

$\Rightarrow)$ It is well known.

\end{proof}

%
%
%

Let $\mathbb M$ and $\mathbb M'$ be $\mathcal R$-modules. If $\mathbb M$ preserves direct limits, then $\Hom_{\mathcal R}(\mathbb M,\mathbb M')$ is a set:
Choose a set $A$ of representatives of the isomorphism classes of finitely presented $R$-modules. Any morphism $f\colon \mathbb M\to \mathbb M'$ is determined
by the set $\{f_S\}_{S\in A}$ since given an $R$-module $T$ we can write
$T=\ilim{i\in I} S_i$, where $S_i\in A$ for any $i\in I$, and the diagram 
$$\xymatrix{ \mathbb M(\ilim{i} S_i)=\mathbb M (S) \ar[r]^-{f_S} & \mathbb M' (S)=\mathbb M'(\ilim{i} S_i)\\ \ilim{i}\mathbb M( S_i) \ar@{=}[u] \ar[r]^-{[f_{S_i}]} & 
\ilim{i}\mathbb M'( S_i) \ar[u]}$$
is commutative. Therefore, $\Hom_{\mathcal R}(\mathbb M,\mathbb M')\subset \prod_{S\in A} \Hom_{gr}(\mathbb M(S),\mathbb M'(S))$.

\begin{proposition} \label{4.8} If $\mathbb M_1$ and  $\mathbb M_2$ preserve direct limits  and $f\colon \mathbb M_1\to\mathbb M_2$ is a morphism of $\mathcal R$-modules,
then $\Ker f$, $\Ima f$  and $\Coker f$ preserve direct limits.\end{proposition}

\begin{proposition} \label{4.85} If $\{\mathbb M_i\}_{i\in I}$ is a direct system 
of $\mathcal R$-modules that preserve direct limits, then $\ilim{i}\mathbb M_i$ preserves direct limits.
\end{proposition}

 Choose a set $A$ of representatives of the isomorphism classes of finitely presented $R$-modules. $\{\mathcal P^*\}_{P\in A}$ is a family of generators of the category of $\mathcal R$-modules that preserve direct limits: 
Let $\mathbb M$ and $\mathbb N$ be  $\mathcal R$-modules that preserve direct limits and suppose that $\mathbb N\underset\neq\subset \mathbb M$. Then, there exist an $R$-module $S$ and $m\in \mathbb M(S)$, such that $m\notin \mathbb N(S)$. $S=\ilim{i\in I} P_i$ is a direct limit of finitely presented $R$-modules $P_i\in A$, and
$\mathbb M(S)=\ilim{i} \mathbb M(P_i)$. Hence, there exist $i\in I$ and $m_i\in 
\mathbb M(P_i)$ such that $m_i$ is mapped to $m$ by the morphism
$\mathbb M(P_i)\to \mathbb M(S)$. Obviously, $m_i\notin \mathbb N(P_i)$.
Finally, observe that $\Hom_{\mathcal R}(\mathcal P_i^*,\mathbb N)=\mathbb N(P_i)
\underset\neq\subset \mathbb M(P_i)=\Hom_{\mathcal R}(\mathcal P_i^*,\mathbb M)$.

Therefore, $\mathbb U:=\oplus_{P\in A} \mathcal P^*$ is a generator of the category
of $\mathcal R$-modules that preserve direct limits.

\begin{theorem}  \label{4.6N} Let $\mathbb M$ be an $\mathcal R$-module.
$\mathbb M$ preserves direct limits   iff there exists an exact sequence of morphisms of $\RR$-modules $$\oplus_{i\in I} \mathcal Q_i^*\to \oplus_{j\in J} \mathcal P_j^*\to\mathbb M\to 0,$$ where $P_i,Q_j$ are finitely presented $R$-modules, for each $i\in I$ and $j\in J$.
\end{theorem}

\begin{proof} $\Leftarrow)$ $\oplus_i \mathcal P_i^*$ and $\oplus_{j} \mathcal Q_j^*$
preserve direct limits by \ref{1.9} and \ref{4.85}. $\mathbb M$ preserves direct limits by \ref{4.8}.

$\Rightarrow)$ Put $I:=\Hom_{\mathcal R}(\mathbb U,\mathbb M)$. Hence, the natural morphism $$\pi\colon \oplus_{I} \mathbb U\to \mathbb M$$ is an epimorphism. By Proposition \ref{4.8},  $\Ker \pi$ preserve direct limits. Again, there exists an epimorphism $\oplus_J \mathbb U\to \Ker \pi$, and we conclude.

\end{proof}

\begin{definition} We will denote by $\langle\text{\sl Qs-ch}\rangle$ the full subcategory of the category of $\RR$-modules whose objects are the $\RR$-modules that preserve direct limits.
\end{definition}

\begin{theorem} $\langle\text{\sl Qs-ch}\rangle$ is the smallest full subcategory of the category of $\RR$-modules stable by kernels, cokernels, direct limits
and isomorphisms (if an $\RR$-module is isomorphic to an object of the subcategory then
it belongs to the subcategory)
that contains the $\mathcal R$-module $\mathcal R$ (or quassi-coherent modules).
\end{theorem}

\begin{proof}  Quassi-coherent modules preserve direct limits. 
By Theorem \ref{4.6N}, Proposition \ref{4.8} and Proposition \ref{4.85} we have only to prove that $\mathcal P^*\in \langle\text{\sl Qs-ch}\rangle$, for any finitely presented $R$-module $P$. Consider an exact sequence of $R$-module morphisms $R^n\to R^m\to P\to 0$. Dually, $0\to \mathcal P^*\to \mathcal R^m\to \mathcal R^n$ is exact and $\mathcal P^*\in \langle\text{\sl Qs-ch}\rangle$.

\end{proof}

\begin{corollary} Let $K$ be a field and $\mathbb M$ an $\mathcal K$-module.  $\mathbb M\in \langle\text{\sl Qs-ch}\rangle$  iff $\mathbb M$ is quasi-coherent.

\end{corollary}

\begin{corollary} \label{C5.3}
If $\mathbb M\in {\langle\text{\sl Qs-ch}\rangle}$ then $\mathbb M^*$ preserves direct products.
\end{corollary}

\begin{proof} By \ref{4.6N}, there exists an exact sequence of morphisms of $\RR$-modules $\oplus_{i} \mathcal Q_i^*\to \oplus_{j} \mathcal P_j^*\to\mathbb M\to 0$, where $P_i,Q_j$ are finitely presented $R$-modules, for every $i,j$. Taking dual $\RR$-modules, we have
the exact sequence of morphisms
$$0\to \mathbb M^*\to \prod_{j} \mathcal P_j\to \prod_{i} \mathcal Q_i$$
$\mathbb M^*$ preserves direct products since $ \prod_{j} \mathcal P_j$ and  $\prod_{i} \mathcal Q_i$ preserve direct products.

\end{proof}

%

\begin{proposition} \cite[Th 1.]{Watts} \label{Wats} Let $\mathbb M$ be an $\RR$-module. $\mathbb M$ is a quasi-coherent $\RR$-module iff $\mathbb M\in \langle\text{\sl Qs-ch}\rangle$ and it is a right-exact functor.
\end{proposition}

\begin{proof} $\Leftarrow)$ Consider the natural morphism $\mathbb M_{qc}\to \mathbb M$. Given $\oplus_I R$, observe that
$${\mathbb M}_{qc}(\oplus_I R)=\oplus_I {\mathbb M}_{qc}(R)=
\oplus_I {\mathbb M}(R)={\mathbb M}(\oplus_IR).$$
Let $N$ be an $R$-module and let $L_1\to L_2\to N\to 0$ be an exact sequence
of morphisms of $R$-modules, where $L_1$ and $L_2$ are free $R$-modules.
Then, ${\mathbb M}_{qc}(N)={\mathbb M}(N)$ since $\mathbb M$ and ${\mathbb M}_{qc}$ are right exact. Therefore, the natural morphism $\mathbb M_{qc}\to \mathbb M$ is an isomorphism.

\end{proof}

$\oplus_{i\in I} \mathcal P_i^*$ is a projective $\RR$-module, since
$$\Hom_{\RR}(\oplus_{i\in I} \mathcal P_i^*,\mathbb M)=\prod_{i\in I} \mathbb M(P_i),\text{ for any  $\RR$-module }\mathbb M.$$ Then, $\langle\text{\sl Qs-ch}\rangle$ has enough projective $\RR$-modules. $\mathbb U$ is a generator of this category, hence, it has enough injective objects, by  \cite[Theorem  1.10.1]{groty}.

 Let $\mathbb M$ be an $\mathcal R$-module. Consider the obvious morphism
$\pi\colon \oplus_{\Hom_{\mathcal R}(\mathbb U,\mathbb M)}\mathbb U \to \mathbb M$. Observe that $\pi_P$ is surjective for any finitely presented $R$-module $P$. Likewise, consider the obvious morphism  $\pi'\colon \oplus_{\Hom_{\mathcal R}(\mathbb U,\Ker\pi)}\mathbb U \to \Ker \pi$. 
Again, $\pi'_P$ is surjective for any finitely presented $R$-module $P$.
Let $\phi$ be the composite morphism
$$\oplus_{\Hom_{\mathcal R}(\mathbb U,\Ker\pi)}\mathbb U \overset{\pi'}\longrightarrow \Ker \pi\subset \oplus_{\Hom_{\mathcal R}(\mathbb U,\mathbb M)}\mathbb U $$
Put $\mathbb M_{\langle\text{\sl Qs-ch}\rangle} := \Coker \phi$. Observe that there is a natural morphism
 $\mathbb M_{\langle\text{\sl Qs-ch}\rangle} \overset{i_{\mathbb M}}\to \mathbb M$ and that $\mathbb M_{\langle\text{\sl Qs-ch}\rangle}\in \langle\text{\sl Qs-ch}\rangle. $ Besides, $$\mathbb M_{\langle\text{\sl Qs-ch}\rangle}(P)=\mathbb M(P),$$ 
for any finitely presented $R$-module $P$. Hence, if $\mathbb M\in {\langle\text{\sl Qs-ch}\rangle}$, then $\mathbb M_{\langle\text{\sl Qs-ch}\rangle}=\mathbb M$.

Observe that the assignation $\mathbb M \rightsquigarrow \mathbb M_{\langle\text{\sl Qs-ch}\rangle}$ is functorial, that is, given a morphism a morphism $f\colon \mathbb N\to \mathbb M$, we can define a natural morphism $f_{\langle\text{\sl Qs-ch}\rangle} \colon \mathbb N_{\langle\text{\sl Qs-ch}\rangle}\to
\mathbb M_{\langle\text{\sl Qs-ch}\rangle}$. Besides, the diagram
$$\xymatrix{\mathbb N \ar[rr]^-{f}  & &\mathbb M\\ \mathbb N_{\langle\text{\sl Qs-ch}\rangle} \ar[rr]^-{f_{\langle\text{\sl Qs-ch}\rangle}} \ar[u]^-{i_{\mathbb N}} & & \mathbb M_{\langle\text{\sl Qs-ch}\rangle} \ar[u]_-{i_{\mathbb M}}  }$$
is commutative. Obviously, we have the following proposition.

\begin{proposition} \label{AdjQs}
The functorial morphism
$$ \Hom_{\mathcal R}(\mathbb N,\mathbb M_{\langle\text{\sl Qs-ch}\rangle})\to \Hom_{\mathcal R}(\mathbb N,\mathbb M),\,\, f \mapsto i_{\mathbb M}\circ f,$$
is an isomorphism, for any $\mathcal R$-module $\mathbb N\in \langle\text{\sl Qs-ch}\rangle$. 
\end{proposition}
 
 If 
$0\to \mathbb M'\to \mathbb M\to \mathbb M''\to 0$
is an exact sequence of morphisms of $\mathcal R$-modules, then 
$$0\to \mathbb M'_{\langle\text{\sl Qs-ch}\rangle}\to \mathbb M_{\langle\text{\sl Qs-ch}\rangle}\to \mathbb M''_{\langle\text{\sl Qs-ch}\rangle}\to 0$$
is an exact sequence of morphisms of $\mathcal R$-modules.

Every $R$-module $M$ is functorially a direct limit of finitely presented $R$-modules:
Put $I=M$ and let $\pi\colon \oplus_I R\to M$ be the obvious epimorphism.
For each finite subset $J\subseteq I$ let $\pi_J$  be the obvious composition  $\oplus_J R\hookrightarrow \oplus_I R\overset\pi\to M$. Let $K_M$ be the set of pairs  $(J,N)$, where $J$ is  a finite subset of $I$ and $N$ is a finite submodule of $\Ker \pi_J$. $K_M$ is a directed set: $(J,N)\leq (J',N')$ if $J\subseteq J'$ and  $N\subseteq N'$. Given $(J,N)$, $(J',N')$, let $J'':=J\cup J'$ and $N'':=N+N'$, then$(J,N)$,$(J',N')\leq (J'',N'')$. It is easy to check that $M=\ilim{(J,N)\in K_M} (\oplus_JA)/N$. Let us denote $ (\oplus_JA)/N=P_{(J,N)}$.

Let $\bf Funct$ be the category of covariant and additive functors from the category of finitely presented $R$-modules to the category of abelian groups.
Given $\mathbb G\in \bf Funct$, 
let $i_*\mathbb G$ be the $\RR$-module defined by  $i_*\mathbb G(M)=\ilim{k\in K_M} \mathbb G(P_k)$. Given an $\RR$-module $\mathbb M$ let $i^*\mathbb M\in \bf Funct$ be defined by 
$i^*\mathbb M(P):=\mathbb M(P)$. Reader can check that $i^*$ and $i_*$ stablish a categorical equivalence between  $\bf Funct$ and $\langle\text{\sl Qs-ch}\rangle$.


\section{FP modules}

\begin{definition} An $\mathcal R$-module $\mathbb M\in \langle\text{\sl Qs-ch}\rangle$ is said to be an F module if  the natural morphism $\mathbb M(\prod_{i\in I} S_i)\to \prod_{i\in I} \mathbb M(S_i)$ is an epimorphism, for any set $\{S_i\}_{i\in I}$ of $R$-modules.

\end{definition}

\begin{proposition} \label{P4.2} $\mathbb M\in \langle\text{\sl Qs-ch}\rangle$ is an F module iff there exists an epimorphism
$\mathcal P^*\to \mathbb M$, where $P$ is a finitely presented $R$-module.

\end{proposition}

\begin{proof} $\Rightarrow)$ There exists an epimorphism $\pi\colon \oplus_j \mathcal Q_j^*\to \mathbb M$, by Theorem \ref{4.6N}. Put $W:=\prod_j Q_j$. Consider the projection $W\to Q_j$, then we have the natural morphism $\mathcal Q_j^*\to \mathcal W^*$ and the morphism $\oplus_j \mathcal Q_j^*\to \mathcal W^*$.
The composite map
$$\Hom_{\RR}(\mathcal W^*,\mathbb M)\!=\!{\mathbb M}(\prod_j Q_j)\!\to \!\prod_j {\mathbb M}(Q_j)\!=\!\prod_j \Hom_{\RR}(\mathcal Q_j^*,\mathbb M)\!=
\Hom_{\RR}(\oplus_j \mathcal Q_j^*,\mathbb M)\!,$$
is surjective. Then, $\pi$ factors through an epimorphism $\pi'\colon\mathcal W^*\to \mathbb M$. Let $\{P_i\}$ be a direct system of finitely presented $R$-modules such that $W=\ilim{i} P_i$. Then,
$$\Hom_{\RR}(\mathcal W^*,\mathbb M)={\mathbb M}(W)=\ilim{i}
{\mathbb M}(P_i)=\ilim{i}\Hom_{\RR}(\mathcal P_i^*,\mathbb M),$$
and $\pi'$ factors through an epimorphism $f\colon \mathcal P_i^*\to \mathbb M$.  

\end{proof}

\begin{example} $M$ is a finite module iff $\mathcal M$ is an F module.\end{example}

\begin{example} Module schemes preserve direct products:
$${\mathcal N^*}(\prod_{i\in I} S_i)=\Hom_{R}(N,\prod_{i\in I} S_i)=\prod_{i\in I} \Hom_{R}(N,S_i)=\prod_{i\in I} {\mathcal N^*}(S_i).$$
\end{example}

\begin{proposition} \label{4.4b} A quasi-coherent module $\mathcal M$ preserves direct products iff $M$ is a finitely presented (right) $R$-module.\end{proposition}

\begin{proposition} \label{5.3} If  $\mathbb M_1$ and $\mathbb M_2$ preserve direct products and $f\colon \mathbb M_1\to \mathbb M_2$ is a morphism of $\mathcal R$-modules, then $\Ker f $, $\Ima f$ and $\Coker f$ preserve direct products.
\end{proposition}

\begin{definition} \label{sFP} An $\mathcal R$-module $\mathbb F\in {\langle\text{\sl Qs-ch}\rangle}$ will be said to be an FP module if it preserves direct products.\end{definition}

\begin{example} Let $P$ be a finitely presented $R$-module.
Then, $\mathcal P$ and $\mathcal P^*$ are FP modules.
\end{example}

\begin{proposition} \label{C5.14}
Suppose that $\mathbb F_1$ and $\mathbb F_2$ are FP modules.  If $f\colon \mathbb F_1\to\mathbb F_2$ is a morphism of $\mathcal R$-modules, then  $\Ker f$, $\Ima f$ and $\Coker f$ are FP modules.
\end{proposition}

\begin{proposition} \label{4.11} Let $\mathbb F$ be an $\mathcal R$-module. $\mathbb F$ is an FP module  iff there exists an exact sequence of $\RR$-modules $\mathcal P^*\to \mathcal Q^*\to \mathbb F\to 0$, where $P$ and $Q$ are finitely presented $R$-modules.\end{proposition}

\begin{proof} $\Leftarrow)$ It is an immediate consequence of Proposition \ref{C5.14}.

$\Rightarrow)$ There exists an epimorphism $\pi\colon \mathcal Q*\to \mathbb F$,
 where $Q$ is a finitely presented $R$-module, by Proposition \ref{P4.2}.  $\Ker \pi$ is an FP module by Proposition \ref{C5.14}.
Again, there exists an epimorphism $\mathcal P^*\to \Ker \pi$, for some finitely presented $R$-module $P$. We are done.

\end{proof}

\begin{proposition} Let $\mathbb F$ be an $\mathcal R$-module. $\mathbb F$ is a projective and FP module iff $\mathbb F\simeq \mathcal P^*$ for a finitely presented $R$-module $P$.

\end{proposition}

\begin{proof} $\Rightarrow)$ By Proposition \ref{4.11}, there exists an epimorphism
$\mathcal Q^*\to \mathbb F$, where $Q$ is a finitely presented $R$-module.
Then, $\mathcal Q^*\simeq \mathbb F\oplus \mathbb G$, for some $\mathcal R$-module $\mathbb G$, since $\mathbb F$ is projective. 
In particular, $\mathbb F$ is reflexive since $\mathcal Q^*$ is reflexive. Taking dual modules, 
$\mathcal Q\simeq \mathbb F^*\oplus \mathbb G^*$. Then, 
$$(\mathbb F^*)_{qc}\oplus {\mathbb G^*}_{qc}\simeq \mathcal Q_{qc}=\mathcal Q\simeq \mathbb F^*\oplus \mathbb G^*$$ 
Therefore,  $(\mathbb F^*)_{qc}=\mathbb F^*$ and $(\mathbb F^*)_{qc}(R)$ is a finitely presented $R$-module since it is a direct summand of $Q$. Finally, 
$\mathbb F=((\mathbb F^*)_{qc})^*$.

\end{proof}

\begin{proposition} \label{P5.5} If $\mathbb F$ is an FP module then $\mathbb F^*$ is an FP (right) module.

\end{proposition}

\begin{proof}  There exists an exact sequence of morphisms
 of $\RR$-modules $\mathcal P^*\to \mathcal Q^*\to \mathbb F\to 0$, where $P$ and $Q$ are finitely presented $R$-modules. Taking dual modules, we obtain the exact sequence $0\to\mathbb F^*\to\mathcal Q\to\mathcal P$.
 Hence, $\mathbb F^*$ is an FP (right) module by Proposition \ref{C5.14}.

 \end{proof}

\begin{proposition} 
\cite[7.14]{Matsumura} \label{1.4} If $0\to \mathcal M_1\overset{i'}\to \mathcal M_2\overset{\pi'}\to\mathcal M_3\to 0$ is an exact sequence of morphisms of $\mathcal R$-modules and $M_3$ is a finitely presented module, then this exact sequence splits.
\end{proposition}

\begin{corollary} \label{1.4b} Let $P$ be a finitely presented $R$-module and $M$ an $R$-module. Then,
$$\Ext_{\RR}^1(\mathcal P,\mathcal M)=0.$$ 
\end{corollary}

\begin{proof} If 
$$0\to \mathcal M \to \mathbb M\to \mathcal P\to 0$$
is an exact sequence of morphisms of $\RR$-modules, then $\mathbb M$ is quasi-coherent
since $\mathbb M\in {\langle\text{\sl Qs-ch}\rangle}$ and it is right exact. By Proposition \ref{1.4}, the sequence of morphisms splits. Hence,  $\Ext_{\RR}^1(\mathcal P,\mathcal M)=0.$

\end{proof}

\begin{lemma} \label{1.1N} Let $f\colon \mathcal V_2\to \mathcal V_1$ be a morphism of $\mathcal R$-modules between quasi-coherent modules. Then, $f$ is an epimorphism iff $f^*\colon \mathcal V_1^*\to \mathcal V_2^*$ is a monomophism.
\end{lemma}

\begin{proof} $\Leftarrow)$ $\Coker f$ is the quasi-coherent module associated with $\Coker f_R$, and $(\Coker f)^*=\Ker f^*=0$. Then,
$\Coker f=(\Coker f)^{**}=0$.

\end{proof}

\begin{corollary} \label{1.4c} Let $P$ be a finitely presented $R$-module and $M$ an $R$-module. Then,
$$\Ext_{\RR}^i(\mathcal P,\mathcal M)=0,\text{ for any }i>0.$$ 

\end{corollary}

\begin{proof} Let $R^n\to P$ be an epimorphism and let $\pi\colon \mathcal R^n\to \mathcal P$ be the induced morphism. Observe that $\Ext_{\RR}^1(\mathcal P,\mathcal M)=0$, by Corollary \ref{1.4b} and 
$\Ext_{\RR}^{i+1}(\mathcal P,\mathcal M)=\Ext_{\RR}^{i}(\Ker \pi,\mathcal M)$ for any $i\geq 1$.

$\Ker \pi$ is an FP (right) module, by \ref{C5.14}. There exists an epimorphism $g\colon \mathcal Q^*\to \Ker\pi$, where $Q$ is a finitely presented $R$-module, by \ref{4.11}. Let $g'$ be the composite morphism $\mathcal Q^*\to \Ker \pi\subseteq \mathcal R^n$. $\Ker g=\Ker g'=(\Coker g'^*)^*$, by \ref{trivialon}. $\Coker g'^*$ is equal to the quasi-coherent $\RR$-module associated with $\Coker g'^*_R=:Q'$, which is a finitely presented $R$-module.
We have the exact sequence of morphisms
$$0\to \mathcal Q'^*\to \mathcal Q^*\to \Ker \pi\to 0$$
Then, $\Ext_{\RR}^i(\Ker \pi,\mathcal M)=0$, for $i>1$. Taking dual $\RR$-modules we have the exact sequence of morphisms
 $$0\to (\Ker \pi)^*\to \mathcal Q\to \mathcal Q'\to 0,$$
by \ref{1.1N}. Hence, $\Ext_{\RR}^1(\Ker \pi,\mathcal M)=0$.

\end{proof} 

\begin{theorem} \label{T5.13} Let $\mathbb F$ be an FP module. Then, $\mathbb F$ is reflexive and
$$\Ext_{\RR}^i(\mathbb F,\mathcal M)=0, \text{ for any } i>0 \text{ and  for any (right) $\RR$-module $M$}.$$
\end{theorem}

\begin{proof} By \ref{4.11}, there exists an exact sequence of morphisms $\mathcal P_2^*\overset f\to \mathcal P_1^*\overset g\to \mathbb F\to 0$, where $P_1$ and $P_2$ are finitely prensented $R$-modules.
Taking dual $\RR$-modules, we have the exact sequence of morphisms
$$0\to \mathbb F^*\to \mathcal P_1\overset{f^*}\to \mathcal P_2$$
Put $P_3:=\Coker f^*_R$, which is a finitely presented $R$-module. Then, we have the exact sequence of morphisms of $\RR$-modules
$$0\to \mathbb F^*\to \mathcal P_1\overset{f^*}\to \mathcal P_2\to \mathcal P_3\to 0$$
By Corollary \ref{1.4c}, it is easy to prove that $\Ext_{\RR}^i(\mathbb F^*,\mathcal M)=0,$ for any  $i>0$ and  for any $R$-module $M$. Hence, the sequence of morphisms 
$$0\to \mathcal P_3^* \to \mathcal P^*_2\overset f\to \mathcal P^*_1\to \mathbb F^{**}\to 0$$
is exact and $\mathbb F=\mathbb F^{**}$. Finally, $\mathbb F$ is the dual module of $\mathbb F^*$,  which is an FP (right) module, by \ref{P5.5}. We have just proved that $\Ext_{\RR}^i(\mathbb F,\mathcal M)=0,$ for any $i>0$  and  for any (right) $R$-module $M$.

\end{proof}

\begin{corollary} \label{5.4} Let $\mathbb F_1$, $\mathbb F_2$ and $\mathbb F_3$ be FP modules. If $\mathbb F_1\to \mathbb F_2\to \mathbb F_3$ is an exact sequence of morphisms of $\RR$-modules, then the dual sequence $\mathbb F_3^*\to \mathbb F_2^*\to \mathbb F_1^*$
is exact.

\end{corollary}

\begin{corollary} \label{TH} Let $\mathbb F$ be a right $\mathcal R$-module. $\mathbb F$ is an FP (right) module iff there exists an exact sequence of morphisms of $\RR$-modules
$$0\to \mathbb F\to \mathcal P\to \mathcal Q,$$
 where $P$ and $Q$ are finitely presented $R$-modules.
\end{corollary}

\begin{proof} $\Leftarrow)$ It is an immediate consequence of \ref{C5.14}.

$\Rightarrow)$ $\mathbb F^*$ is an FP module, by \ref{P5.5}. By \ref{4.11}, there exists an exact sequence of morphisms of $\RR$-modules
$\mathcal Q^*\to \mathcal P^*\to \mathbb F^*$, where $P$ and $Q$ are finitely presented $R$-modules. Taking dual $\RR$-modules,
we have the exact sequence $$0\to \mathbb F\overset{\text{\ref{T5.13}}}=\mathbb F^{**}\to \mathcal P\to \mathcal Q.$$
\end{proof}

\begin{corollary} Let $\mathbb M$ be an $\mathcal R$-module.
$\mathbb M$ is an FP module iff it is reflexive and $\mathbb M, \mathbb M^*\in \langle\text{\sl Qs-ch}\rangle$.

\end{corollary}

\begin{proof} $\Rightarrow)$ It is an immediate consequence of Theorem \ref{T5.13}.

$\Leftarrow)$ By Corollary \ref{C5.3}, $\mathbb M^*$ preserves direct products, therefore it is an FP module. By Proposition \ref{P5.5}, $\mathbb M=\mathbb M^{**}$ is an FP module.

\end{proof}

\begin{lemma} \label{L6.6} Let $\mathbb F$ be an FP module and $\{\mathbb M_i\}$ a direct system of $\RR$-modules. Then,
$\Hom_{\RR}(\mathbb F,\ilim{i} \mathbb M_i)=\ilim{i}
\Hom_{\RR}(\mathbb F,\mathbb M_i).$\end{lemma}

\begin{proof} By \ref{4.11}, there exists an exact sequence of morphisms $\mathcal P^*\to \mathcal Q^*\to \mathbb F\to 0$, where $P$ and $Q$ are finitely presented $R$-modules. By \ref{L5.11}, $\Hom_{\RR}(\mathcal P^*,\ilim{i} \mathbb M_i)=\ilim{i}
\Hom_{\RR}(\mathcal P^*,\mathbb M_i)$, for any finitely presented $\RR$-module $P$. Now it is easy to prove that $\Hom_{\RR}(\mathbb F,\ilim{i} \mathbb M_i)=\ilim{i}
\Hom_{\RR}(\mathbb F,\mathbb M_i).$

\end{proof}

\begin{theorem} \label{T5.17} Let $\mathbb M$ be an $\RR$-module. $\mathbb M\in \langle\text{\sl Qs-ch}\rangle$ iff it is a direct limit of FP modules.
\end{theorem}

\begin{proof}
$\Rightarrow)$ By \ref{4.6N}, there exists an exact sequence of morphism of $\mathcal R$-modules
$$\oplus_{i\in I} \mathcal P_i^*\to \oplus_{j\in J} \mathcal Q_j^*\to\mathbb M\to 0,$$ where $P_i,Q_j$ are finitely presented $R$-modules, for any $i,j$.
Let $F$ (respectively $G$) be the set of all finite subsets of $I$ (respectively $J$). By \ref{L6.6}, given $I'\in F$ there exists $J'\in G$ such that the composite morphism
$\oplus_{i\in I'} \mathcal P_i^*\hookrightarrow \oplus_{i\in I} \mathcal P_i^*\to \oplus_{j\in J} \mathcal Q_j^*$ factors through $\oplus_{j\in J'} \mathcal Q_j^*$, since $\oplus_{i\in I'} \mathcal P_i^*=(\oplus_{i\in I'} \mathcal P_i)^*$ is an FP module. We will say that $J'\geq I'$ and we will denote
$$\mathbb F_{J'\geq I'} :=\Coker[\oplus_{i\in I'} \mathcal P_i^*\to \oplus_{j\in J'} \mathcal Q_j^*].$$
Let $H$ be the set of pairs $(J',I')$, where $J'\in G$, $I'\in F$ and $J'\geq I'$.
Now, it is easy to check that $\mathbb M=\ilim{(J',I')\in H} \mathbb F_{J'\geq I'}.$

$\Leftarrow)$ It is an immediate consequence of \ref{4.85}.

\end{proof}

\begin{corollary} $\mathbb M\in \langle\text{\sl Qs-ch}\rangle$ is an FP module iff $$\Hom_{\mathcal R}(\mathbb M,\ilim{i\in I} \mathbb M_i)=\ilim{i\in I}\Hom_{\mathcal R}(\mathbb M,\mathbb M_i)$$
for any direct system of $\mathcal R$-modules $\{\mathbb M_i\}_{i\in I}$.
\end{corollary}

\begin{proof}  $\Rightarrow$) It is Lemma \ref{L6.6}.

$\Leftarrow$)  By Theorem \ref{T5.17}, $\mathbb M=\ilim{i\in I} \mathbb F_i$, where $\mathbb F_i$ is a FP modules, for any $i$. The identity morphism
$\mathbb M \to \mathbb M=\ilim{i\in I} \mathbb F_i$ factors through a morphism 
$\mathbb M\to \mathbb F_i$, for some $i\in I$. Then, $\mathbb M$ is a direct summand of $\mathbb F_i$. By Proposition \ref{C5.14}, $\mathbb M$ is an FP module.

\end{proof}

\begin{proposition} 
Let $I$ be an injective $R$-module. Then,
$$\Hom_{\RR}(\mathbb M,\mathcal I)=\Hom_{R}(\mathbb M(R),I),$$
for any $\RR$-module $\mathbb M\in \langle\text{\sl Qs-ch}\rangle$. In particular,
$\mathcal I$ is an injective object of $\langle\text{\sl Qs-ch}\rangle$.
\end{proposition}

\begin{proof} Let  $\mathbb F$ be an FP $\RR$-module. By Corollary \ref{TH}, there exists an exact sequence of $\RR$-module morphisms
$$0\to \mathbb F\to \mathcal P\to\mathcal Q$$
where $P$ and $Q$ are finitely presented $R$-modules. By Corollary \ref{5.4},
we have the exact sequence morphisms of groups
$$\Hom_{\mathcal R}(\mathcal Q,\mathcal I)\to \Hom_{\mathcal R}(\mathcal P,\mathcal I)\to \Hom_{\mathcal R}(\mathbb F,\mathcal I)\to 0$$
On the other hand,  
$$\Hom_{R}(Q,I)\to \Hom_{R}( P,I)\to \Hom_{R}(\mathbb F(R),I)\to 0$$ is exact, because $I$ is an injective $R$-module. Hence,  $\Hom_{\mathcal R}(\mathbb F,\mathcal I)=\Hom_{ R}(\mathbb F(R),I)$.
If $\mathbb M'\in \langle\text{\sl Qs-ch}\rangle$, then $\mathbb M'=\ilim{i} \mathbb F_i$, where $\mathbb F_i$ are FP $\RR$-modules, by Theorem \ref{T5.17}. Then,
$$\Hom_{\mathcal R}(\mathbb M',\mathcal I)=\plim{i} \Hom_{\mathcal R}(\mathbb F_i,\mathcal I)=\plim{i} \Hom_{R}(\mathbb F_i(R),\mathcal I)=
\Hom_{R}(\mathbb M'(R),I)$$

\end{proof}

\begin{definition} An $R$-module $M$ is said to be pure-injective if for any pure morphism $N\hookrightarrow N'$ the induced morphism 
$$\Hom_R(N',M)\to \Hom_R(N,M)$$
is surjective.
\end{definition}

\begin{proposition} \label{pure} $\mathbb M$ is an injective object of $\langle\text{\sl Qs-ch}\rangle$ iff $\mathbb M$ is the  quasi-coherent $\RR$-module associated with a pure-injective $R$-module.
\end{proposition}

\begin{proof} $\Rightarrow)$ First, let us prove that $\mathbb M$ is quassi-coherent.
By Proposition \ref{Wats}, we have to prove that $\mathbb M$ is a right-exact functor.
Let $N_1\overset f\to N_2\overset g\to N_3\to 0$ be an exact sequence of $R$-module morphisms. We have to prove that the sequence of morphisms $\mathbb M(N_1)\to \mathbb M(N_2)\to \mathbb M(N_3)\to 0$ is exact.
Put $N_1:=\ilim{i} P_i$, where  $P_i$ are finitely presented $R$-modules and $N_{3i}:=N_2/f(P_i)$.
We have the exact sequences $P_i\to N_2\to N_{3i}\to 0$ and
$\ilim{i} P_i=N_1$ and $\ilim{i} N_{3i}=N_3$. We have only to prove that 
that the sequence of morphisms $\mathbb M(P_i)\to \mathbb M(N_2)\to \mathbb M(N_{3i})\to 0$ is exact. That is, we can suppose that $N_1$ is a finitely presented $R$-module. Put $N_2=\ilim{j} P_j$. The morphism $f\colon N_1\to N_2$ factors through a morphism $f_j\colon N_1\to P_j$. 
Let $f_{j'}$ be the composite morphism $N_1\overset{f'}\to N_j\to N_{j'}$
for any $j'\geq j$.
Put $N_{3j'}:=N_2/f_{j'}(N_1)$. We have the exact sequences $N_1\to P_{j'}
\to N_{3j'}\to 0$ and
$\ilim{j'>j} P_{j'}=N_2$ and $\ilim{j'>j} N_{3j'}=N_3$. 
We have only to prove that 
that the sequence of morphisms $\mathbb M(N_1)\to \mathbb M(P_{j'})\to \mathbb M(N_{3j'})\to 0$ is exact. That is, we can suppose that
$N_1$, $N_2$ and $N_3$ are finitely presented $R$-modules.
The sequence of $\RR$-module morphisms $0\to \mathcal N_3^*\to \mathcal N_2^*\to \mathcal N_1^*$  is an exact sequence. Then, taking $\Hom_{\RR}(-,\mathbb M)$
the sequence
$$\mathbb M (N_1)\to \mathbb M (N_2)\to \mathbb M(N_3)\to 0$$
is exact. We are done.

Put $\mathbb M:=\mathcal M$.
If $N\hookrightarrow N'$ is a pure morphism, then the induced morphism 
$\mathcal N\hookrightarrow \mathcal N'$ is a monomorphism. Then, the morphism 
$$\Hom_R(N',M)=\Hom_{\RR}(\mathcal N',\mathcal M)\to \Hom_{\RR}(\mathcal N,\mathcal M)=\Hom_R(N,M)$$
is surjective and $M$ is pure-injective.

$\Leftarrow)$ Let $M$ be pure-injective. $\langle\text{\sl Qs-ch}\rangle$ has enough injective modules. Let $i\colon \mathcal M\hookrightarrow \mathcal M'$ be a monomorphism, where
$\mathcal M'$ is an injective object of $\langle\text{\sl Qs-ch}\rangle$. The morphism $i$ has a retraction since $M$ is pure-injective. $\mathcal M$ is an injective object of $\langle\text{\sl Qs-ch}\rangle$ since it is a direct summand of $ \mathcal M'$.

\end{proof}

\section{Mittag-Leffler modules}

Mittag-Leffler conditions were first introduced by Grothendieck in \cite{Grot}, and deeply studied
by some authors, such as  Raynaud and Gruson in \cite{RaynaudG}.  Recently, Drinfeld suggested to employ them in infinite dimensional algebraic geometry (see \cite{Drinfeld}  and \cite{Herbera})

\begin{definition} We will say that an $\RR$-module $\mathbb M\in \langle\text{\sl Qs-ch}\rangle$ is an ML module if  the natural morphism $\mathbb M(\prod_{i\in I} S_i)\to \prod_{i\in I} \mathbb M(S_i)$ is injective for any set $\{S_i\}_{i\in I}$ of  $R$-modules.\end{definition}

\begin{example} FP modules are ML modules\end{example}

\begin{example} $M$ is an Mittag-Leffler module iff $\mathcal M$ is a ML module(\cite[Tag 059H]{stacks-project}).\end{example}

\begin{proposition} \label{P6.3} Let $\mathbb M$ and $\mathbb M'$ be   ML modules and $f\colon \mathbb M\to \mathbb M'$ a morphism of $\RR$-modules. Then, $\Ker f$ and $\Ima f$ are ML modules.
\end{proposition}

\begin{proof} $\Ker f$ and $\Ima f$ preserve direct limits, by \ref{4.8}.
If $\mathbb F'$ is an $\RR$-submodule of an ML module $\mathbb F$, then the morphism $\mathbb F'(\prod_i S_i)\to \prod_i \mathbb F'(S_i)$ is injective for any set of $\{S_i\}_{i\in I}$ $R$-modules. Hence,
$\Ker f$ and $\Ima f$ are ML modules.

\end{proof}

\begin{lemma} \label{L6.4} Let $\mathbb M$ be an ML module, $\mathbb F$ an FP module and $f\colon \mathbb F\to \mathbb M$ a morphism of $\RR$-modules. Then, $\Ker f$ and $\Ima f$ are  FP modules.
\end{lemma}

\begin{proof}  $\Ker f$ preserves direct limits by \ref{4.8}. Let $\{S_i\}_{i\in I}$ be a set of $R$-modules. Consider the commutative diagram with exact rows
$$\xymatrix{ 0 \ar[r] &  \Ker f(\prod_{i\in I} S_i) \ar[d] \ar[r] & \mathbb F(\prod_{i\in I} S_i) \ar[r]  \ar@{=}[d] & \mathbb M(\prod_{i\in I} S_i) \ar@{^{(}->}[d] \\ 0 \ar[r] &  \prod_{i\in I}  \Ker f( S_i)  \ar[r] & \prod_{i\in I} \mathbb F(S_i) \ar[r] & \prod_{i\in I} \mathbb M( S_i)}$$
Hence, $\Ker f(\prod_{i\in I} S_i)=\prod_{i\in I}\Ker f( S_i)$ and $\Ker f $ is an FP module. $\Ima f$ is  isomorphic to the cokernel of the monomorphism $\Ker f \to \mathbb F$, which is an FP module by \ref{C5.14}.

\end{proof}

\begin{lemma} \label{L6.3} If $\{\mathbb M_i,f_{ij}\}$ is a direct system of ML modules and $f_{ij}$ is a monomorphism for any $i\leq j$, then $\ilim{i} \mathbb M_i$ is an ML module.\end{lemma}

\begin{proof} $\ilim{i} \mathbb M_i\in \langle\text{\sl Qs-ch}\rangle$. Besides, the composite morphism
$$\aligned   & (\ilim{i} \mathbb M)(\prod_j S_j)=\ilim{i}
\mathbb M_i(\prod_j S_j)\hookrightarrow \ilim{i} \prod_j \mathbb M_i(S_j)\\ & \hookrightarrow  \prod_j \ilim{i}
\mathbb M_i(S_j) = \prod_j  (\ilim{i} \mathbb M_i)(S_j)\endaligned$$
is injective, for any set $\{S_j\}_{j\in J}$ of $R$-modules.

\end{proof}

\begin{proposition} \label{P6.5} An $\RR$-module $\mathbb M$ is an ML module iff
$\mathbb M$ is a direct limit of  FP submodules.
\end{proposition}

\begin{proof} $\Rightarrow)$ By \ref{4.6N}, there exists an epimorphism $\pi\colon \oplus_{i\in I} \mathcal P_i^*\to \mathbb M$.  Let $F$ be the set of all finite subsets of $I$. Given $J\in F$,
put $\mathbb F_J:= \pi(\oplus_{i\in J} \mathcal P_i^*)$, which is an FP module by \ref{L6.4}. Obviously, $\mathbb M=\ilim{J\in F} \mathbb F_J$.

$\Leftarrow)$ It is an immediate consequence of \ref{L6.3}.
\end{proof}

\begin{proposition} \label{6.8} $\mathbb M\in \langle\text{\sl Qs-ch}\rangle$ is  an ML module iff for every  FP module $\mathbb F$ and every morphism $f\colon \mathbb F\to \mathbb M$, $\Ima f$ is an FP module.
\end{proposition}

\begin{proof} $\Rightarrow)$ It is Lemma \ref{L6.4}.


$\Leftarrow)$ By \ref{4.6N}, there exists an epimorphism $\pi\colon \oplus_{i\in I} \mathcal P_i^*\to \mathbb M$.  Let $F$ be the set of all finite subsets of $I$. Given $J\in F$,
put $\mathbb F_J:= \pi(\oplus_{i\in J} \mathcal P_i^*)$, which is an FP module. Obviously, $\mathbb M=\ilim{J\in F} \mathbb F_J$. By \ref{P6.5},
$\mathbb M$ is  an ML module.

\end{proof}

We can now generalize a crucial closure property of the category of Mittag-Leffler modules (see \cite[Prop 2.2]{Herbera})

\begin{corollary} If $\{\mathbb M_i,f_{ij}\}_{i,j\in I}$ is a direct system of ML modules and  $\ilim{n\in\mathbb N} \mathbb M_{i_n}$ is an ML module
for any ordered subset  $\{i_1\leq i_2\leq \cdots\leq i_n\leq\cdots\}$ of $I$,
then $\ilim{i\in I} \mathbb M_i$ is an ML module.
\end{corollary}

\begin{proof} Let $\mathbb F$ be an FP module and $f\colon \mathbb F\to \ilim{i} \mathbb M_i$ a morphism of $\RR$-modules. By Lemma \ref{L6.6}, $f$ factors though a morphism $f_i\colon \mathbb F\to  \mathbb M_i$. Put $f_j:=f_{ij}\circ f_i$ for any $j\geq i$ and put 
$\mathbb K_j:=\Ker f_j$. Observe that $\mathbb  K_j\subseteq \mathbb K_{j'}$, for any $j'\geq j\geq i$. Let  
$$i\leq {j_1}<{j_2}< \cdots <{j_n}< \cdots$$
be an ordered subset of $I$. Observe that $\mathbb K':=\cup_{n\in \mathbb N} \mathbb K_{j_n}$ equals the kernel 
of  the natural morphism $\mathbb F\to \ilim{n\in\mathbb N} \mathbb M_{j_n}$. Hence, $\mathbb K'$ is an FP module. The identity morphism 
$Id\colon \mathbb K'\to \mathbb K'$ factors through a morphism $\mathbb K'\to 
\mathbb K_{j_n}$, by Lemma \ref{L6.6}. Then,  $\mathbb K'=
\mathbb K_{j_n}$ and $\mathbb K_{j_n}=\mathbb K_{j_{n+r}}$, for any $r>0$.
Therefore, there exists $j\in I$ such that $\Ker f_j=\Ker f$. Hence,
$\Ima f=\Ima f_j$ which is an FP module. Then,  $\ilim{i\in I} \mathbb M_i$ is an ML module.

\end{proof}

\begin{theorem} \label{5.8} Let $\mathbb M\in \langle\text{\sl Qs-ch}\rangle$. The following statements are equivalent

\begin{enumerate}

\item $\mathbb M$ is an  ML module.

\item  The kernel of every morphism of $\RR$-modules
$\mathbb F\to\mathbb M$  is  an FP module, for any FP module $\mathbb F$.

\item The kernel of every morphism
$\mathcal N^*\to\mathbb M$ is isomorphic to a quotient of a module scheme, for any  $R$-module $N$.

\item The kernel of every morphism
$\mathcal N^*\to\mathbb M$ preserves direct products, for any  $R$-module $N$.

\end{enumerate}

\end{theorem}

\begin{proof} (1) $\iff$ (2) $\Ima f $ is an FP module iff $\Ker f$ is an FP module, by \ref{C5.14}, and $\Ima f $ is an FP module
iff $\mathbb M$ is an ML module, by \ref{6.8}.

%

(2) $\Rightarrow$ (3) Let $f\colon \mathcal N^*\to\mathbb M$ be a morphism of $\RR$-modules.
There exists a direct system $\{P_i\}$ of finitely presented $R$-modules such that $N=\ilim{i} P_i$. Observe, that
$$\Hom_{\RR}(\mathcal N^*,\mathbb M)={\mathbb M}(N)
=\ilim{i} {\mathbb M}(P_i)=\ilim{i}\Hom_{\RR}(\mathcal P_i^*,\mathbb M).$$
Therefore, $f$ factors through a morphism $g\colon \mathcal P_i^*\to \mathbb M$. $\Ker g$ is an FP module since $\mathcal P_i^*$ is an FP module. There exists an epimorphism $\mathcal Q^*\to \Ker g$
by \ref{4.11}
Consider
the morphism $\pi_1\colon \mathcal N^*\times_{\mathcal P_i^*} \mathcal Q^*\to \mathcal N^*$, $\pi_1(w,v)=w$. It is easy to check that $\Ima \pi_1=\Ker f$. Finally, observe that
$\mathcal N^*\times_{\mathcal P_i^*} \mathcal Q^*=(\mathcal N\oplus_{\mathcal P} \mathcal Q)^*$ and observe that $\mathcal N\oplus_{\mathcal P} \mathcal Q$ is a quasi-coherent module since it is equal to the cokernel of a morphism
${\mathcal P}\to \mathcal N\oplus \mathcal Q$.

(3) $\Rightarrow$ (4) $\Ker[\mathcal N^*\to \mathbb M]\simeq \Ima[\mathcal N'^*\to\mathcal N^*]$, for some
$R$-module $N'$. Hence, $\Ker[\mathcal N^*\to \mathbb M]$ preserves direct products, by \ref{5.3}.

(4) $\Rightarrow$ (1) Let $\mathbb F$ be an FP module and
$f\colon \mathbb F\to \mathbb M$ a morphism of $\RR$-modules
By \ref{4.11}, there exists an epimorphism $\mathcal P^*\to \mathbb F$. Let $g$ be the composite morphism $\mathcal P^*\to \mathbb F\to \mathbb M$.
Then, $\Ima f=\Ima g$.
$\Ima g$ preserves direct limits, by \ref{4.8}.
$\Ker g$ preserves direct products by Hypothesis. $\Ima g=\Coker[\Ker g \to \mathcal P^*]$ preserves direct products, by \ref{5.3}.
Therefore, $\Ima f=\Ima g$ is an FP module and $\mathbb M$ is an ML module, by \ref{6.8}.

\end{proof}

%

\begin{theorem} \label{4.6b} Let $M$ be a right $R$-module. $M$ is a Mittag-Leffler module iff   the image of  every morphism  $f\colon \mathcal R^n\to \mathcal M$ is an FP module, for any $n$.

\end{theorem}

\begin{proof}  $\Rightarrow)$  By \ref{C5.14}, $\Ima  f$  is an FP module.

$\Leftarrow)$  Obviously $\mathcal M$ is a direct limit of FP modules. By Proposition \ref{P6.5}, $\mathcal M$ is an ML module, hence $M$ is a Mittag-Leffler module.


\end{proof}

\begin{corollary} \label{4.6} Let $M$ be a right $R$-module. $M$ is a Mittag-Leffler module iff the image of  every  morphism of $\mathcal R$-modules $f\colon \prod_{\mathbb N} \mathcal R\to \mathcal M$ is an FP module (or preserves direct products).
\end{corollary}

\begin{proof} Any morphism $f\colon \prod_{\mathbb  N} \mathcal R\to \mathcal M$ is the composite morphism of an epimorphism $\prod_{\mathbb  N} \mathcal R\to \mathcal R^n$ and a morphism $g\colon \mathcal R^n\to \mathcal M$, by \ref{prop4}. Observe that $\Ima f=\Ima g$ and $\Ima g$ is an FP module iff it preserves direct products.

\end{proof}

\begin{corollary} Let $M$ be an $R$-module. $M$ is a Mittag-Leffler module iff the kernel of  every  morphism of $\mathcal R$-modules $f\colon \prod_{\mathbb N} \mathcal R\to \mathcal M$ preserves direct products.

\end{corollary}

\begin{proof} It is an immediate consequence of \ref{4.6}, since $\Ker f$ preserves direct products iff $\Ima g$ preserves directs products.

\end{proof}

\section{Strict Mittag-Leffler modules}

\begin{definition} \label{7.1} We will say that an $\RR$-module $\mathbb M\in \langle\text{\sl Qs-ch}\rangle$  is an SML module if it is an ML $\RR$-module and for every  FP submodule $\mathbb F\subseteq \mathbb M$ the dual morphism $\mathbb M^*\to \mathbb F^*$ is an epimorphism.
\end{definition}

\begin{theorem} \label{T7.2} Let $\mathbb M$ be an $\RR$-module. $\mathbb M$ is an SML module iff there exists a direct system $\{\mathbb F_i\}_{i\in I}$
of  FP submodules of $\mathbb M$ such that $\mathbb M=\ilim{i\in I}\mathbb F_i$ and the natural morphism $\mathbb M^*\to \mathbb F_i^*$ is an epimorphism, for each $i\in I$.\end{theorem}

\begin{proof} $\Rightarrow)$ It is an immediate consequence of Proposition \ref{P6.5}.

$\Leftarrow)$ By Proposition  \ref{P6.5}, $\mathbb M$ is an ML module.
Let $\mathbb F\subseteq \mathbb M$ be  an FP submodule. This monomorphism factors through a monomorphism $\mathbb F\to \mathbb F_i$, by Lemma \ref{L6.6}.
Dually, $\mathbb M^*\to \mathbb F_i^*$ is an epimorphism by the hypothesis, and 
$\mathbb F_i^*\to \mathbb F^*$ is an epimorphism by Corollary \ref{5.4}.
Therefore, the morphism $\mathbb M^*\to \mathbb F^*$ is an epimorphism.


\end{proof}

\begin{proposition} \label{ains} A reflexive $\RR$-module $\mathbb M\in \langle\text{\sl Qs-ch}\rangle$ is an SML module if for every  FP (right) module $\mathbb F$ the image of every morphism $f\colon \mathbb M^*\to \mathbb F$ is an FP (right) module.\end{proposition}

\begin{proof} $\Rightarrow)$ Consider the dual morphism $f^*\colon \mathbb F^*\to \mathbb M$. $\Ima f^*$ is an FP module by Proposition \ref{6.8}. Dually, 
$\mathbb F\to (\Ima f^*)^*$ is an epimorphism by  Corollary \ref{5.4}, and the morphism
$(\Ima f^*)^*\to \mathbb M$ is a monomorphism. Hence, $\Ima f=(\Ima f^*)^*$, which is an FP module by Proposition \ref{P5.5}.

$\Leftarrow)$ Let $g\colon \mathbb F\to \mathbb M$ be an $\RR$-module morphism. 
Again, $\Ima g=(\Ima g^*)^*$ which is an FP module. Hence, $\mathbb M$ is an ML $\RR$-module, by Proposition \ref{6.8}. If $g$ is a monomorphism, $\Ima g^*=(\Ima g^*)^{**}=
(\Ima g)^*=\mathbb F^*$, $g^*$ is an epimorphism and $\mathbb M$  is an SML module.

\end{proof}

%
%
%
%
%

\begin{corollary} \label{5.6}  Let $M$ be a right $R$-module. $\mathcal M$ is an SML module iff for every finitely generated submodule $N\subseteq M$ the image ${\tilde N}$ of the associated morphism
$\mathcal N\to \mathcal M$ is an FP module and the morphism
$\mathcal M^*\to {\tilde N}^*$ is an epimorphism.
\end{corollary}

\begin{proof} $\Rightarrow)$ ${\tilde N}$ is an FP module by \ref{4.6b}. 
$\mathcal M^*\to {\tilde N}^*$ is an epimorphism, by \ref{7.1}.

$\Leftarrow)$ It is an immediate consequence of \ref{T7.2}.

\end{proof}

\begin{proposition} \label{7.6} Let $\mathbb M\in \langle\text{\sl Qs-ch}\rangle$ be a reflexive $\RR$-module. $\mathbb M$ is an SML module iff
for any $R$-module $N$ and any morphism $g\colon \mathcal N^*\to \mathbb M$, $\Coker g^*$  is isomorphic to an $\mathcal R$-submodule of a quasi-coherent (right) module.

\end{proposition}

\begin{proof}
$\Rightarrow)$  Let $\{\mathbb F_i\}_{i\in I}$ be a direct system of FP submodules of $\mathbb M$ such that $\mathbb M=\ilim{i} \mathbb F_i$. By \ref{L5.11}, $g$ factors through a morphism
$g_i\colon \mathcal N^*\to \mathbb F_i$, for some $i\in I$. Therefore,
$g^*$ is the composite morphism of the epimorphism $\mathbb M^*\to \mathbb F_i^*$ and $g_i^*\colon \mathbb F_i^*\to \mathcal N$. Hence, $\Ima g^*=\Ima g_i^*$ and $\Coker g^*=\Coker g_i^*$. By \ref{TH}, there exist finitely presented  $R$-modules $P$ and $Q$ and an exact sequence of morphisms $0\to \mathbb F_i^*\overset{r}\to \mathcal P\overset{s}\to \mathcal Q$. By \ref{5.4}, there exists a morphism $t\colon \mathcal P\to \mathcal N$ such that $g_i^*=t\circ r$. Considering the obvious morphisms
$$\mathcal N/\Ima g_i^*\hookrightarrow \mathcal N/\Ima g_i^*\oplus_{\mathcal P/\Ima r} \mathcal Q \simeq \mathcal N\oplus_{\mathcal P} \mathcal Q.$$
we obtain that $\Coker g^*$ is a submodule of a quasi-coherent $\RR$-module.

$\Leftarrow)$ Let $\mathbb F$ be an FP module and
$f\colon \mathbb F\to\mathbb M$ a morphism of $\RR$-modules. Let $f^*\colon \mathbb M^*\to \mathbb F^*$ be  the dual morphism.  We have to prove that $\Ima f^*$ is an FP (right) module.
$\Ima f^*$ preserves direct products since
$\mathbb M^*$ and $\mathbb F^*$ preserve direct products.
There exist an $R$-module $M$ and a monomorphism $\Coker f^*\hookrightarrow \mathcal M$.
Let $g$ be the composite morphism $\mathbb F^*\to \Coker f^*\hookrightarrow \mathcal M$.
$\Ima f^* =\Ker g$, which preserves direct limits, by \ref{4.8}.
Therefore, $\Ima f^*$ is an FP (right) module.

\end{proof}

\begin{proposition} $\mathcal M$ is an SML module iff
the cokernel of every morphism
$\mathcal M^*\to\oplus_{\mathbb N} \mathcal R$  is isomorphic  to an $\mathcal R$-submodule of a quasi-coherent (right) module.
\end{proposition}

\begin{proof} $\Rightarrow)$  The morphism $\mathcal M^*\to\oplus_{\mathbb N} \mathcal R$ factors through a direct summand $\mathcal R^n$ of $\oplus_{\mathbb N} \mathcal R$. $\Coker[\mathcal M^*\to \mathcal R^n]$ is isomorphic  to an $\mathcal R$-submodule of a quasi-coherent (right) module, by \ref{7.6}.

$\Leftarrow)$ Let $f\colon \mathcal R^n\to \mathcal M$ be a morphism of $\RR$-modules and let $f^*\colon \mathcal M^*\to \mathcal R^n$ be the dual morphism. $\Coker f^*$ preserves direct products by \ref{5.3}.
There exist an $R$-module $N$ and a monomorphism $\Coker f^*\hookrightarrow\mathcal N$. $\Coker f^*=\Ima [\mathcal R^n\to \mathcal N]$ preserves direct limits, by \ref{4.8}.  Therefore,  $\Coker f^*$ is an FP (right) module. $\Ima f^*=\Ker[\mathcal R^n\to \Coker f^*]$ is an FP (right) $\RR$-module, by \ref{C5.14}. We have the epimorphism and the monomorphism $\mathcal M^*\twoheadrightarrow \Ima f^*\hookrightarrow \mathcal R^n$. By \ref{5.4}, we have the epimorphism and the monomorphism $$\mathcal R^n \twoheadrightarrow (\Ima f^*)^*\hookrightarrow  \mathcal M$$
Hence, $\Ima f=(\Ima f^*)^*$ is an FP module. By \ref{5.6}, $\mathcal M$ is an SML module.

\end{proof}

\begin{theorem} \label{hard} $\mathbb M\in \langle\text{\sl Qs-ch}\rangle$ is an SML module iff  there exists a monomorphism $\mathbb M\to \prod_{j\in J} \mathcal P_j$, where $P_j$ is a finitely presented (right) module, for each $j\in J$.
\end{theorem}

\begin{proof} $\Rightarrow)$
Choose a set $A$ of representatives of the isomorphism classes of finitely presented (right) $R$-modules. Let $B$ be the set of  pairs
$(\mathcal P, g)$, where $P\in A$ and $g\in \Hom_{\mathcal R}(\mathbb  M,\mathcal P)$. The ``canonical'' morphism
$$G\colon \mathbb M\to \prod _{(\mathcal P,g)\in B}\mathcal P, \, G(m):=(g(m))_{(\mathcal P,g)}$$
is a monomorphism: There exists a direct system $\{\mathbb F_i\}$ of FP submodules of $\mathbb M$, such that
$\mathbb M=\ilim{i\in I} \mathbb F_i$ and the natural morphism $\mathbb M^*\to \mathbb F_i^*$ is an epimorphism, for any $i$.
There exist a finitely presented (right) $R$-module $Q\in A$ and a monomorphism $g_i\colon \mathbb F_i\hookrightarrow \mathcal Q$, by \ref{TH}. There exists $f_i\in
\Hom_{\RR}(\mathbb M,\mathcal Q)$ such that
${f_i}_{|\mathbb F_i}=g_i$, since the morphism 
$$\Hom_{\RR}(\mathbb M,\mathcal Q)={\mathbb M^*}(Q)\to \mathbb F_i^*(Q)=\Hom_{\RR}(\mathbb F_i,\mathcal Q)$$
is surjective.
 Let $\pi_{(\mathcal Q,f_i)}\colon \prod _{(\mathcal P,g)\in B}\mathcal P\to \mathcal Q$ be the projection onto the factor indexed by $(\mathcal Q,f_i)$. Therefore, $G_{|\mathbb F_i}$ is a monomorphism since $\pi_{(\mathcal Q,f_i)}\circ G_{|\mathbb F_i}={f_i}_{|\mathbb F_i}=g_i$. Hence, $G$ is a monomorphism.

$\Leftarrow)$ Let $i\colon \mathbb M\hookrightarrow \prod_{j\in J}\mathcal P_j$ be a monomorphism, where $P_j$ is a finitely presented (right) $R$-module for each $j\in J$.
$\prod_j\mathcal P_j$ preserves direct products.
Hence, ${\mathbb M}(\prod_i S_i) \to \prod_i{\mathbb M}( S_i)$ is injective, for any set $\{S_i\}$ of $R$-modules.
Therefore,  $\mathbb M$ is an ML module. 

Let $\mathbb F$ be an FP $\RR$-module and $\mathbb F\hookrightarrow \mathbb M$  a monomorphism.  By \ref{7.1}, we have to prove that the dual morphism $\mathbb M^*\to \mathbb F^*$ is an epimorphism. We can suppose that $\mathbb M=\prod_j\mathcal P_j. $ We have the commutative diagram  (see Appendix)
 
$$\xymatrix{ \mathbb F\circ \mathbb D  \ar@{^{(}->}[r] \ar@{=}[d]^-{\text{\ref{T-lemma}}} &  \mathbb M\circ  \mathbb D  \ar@{=}[d]^-{\text{\ref{T-lemma}}} \\  \mathbb D\circ \mathbb F^* \ar[r] &  \mathbb D\circ \mathbb M^*}$$
Then, the morphism $\mathbb D\circ \mathbb F^* \to   \mathbb D\circ \mathbb M^*$ is a monomorphism. Hence, the morphism $\mathbb M^*\to \mathbb F^*$ is an epimorphism.

\end{proof}

\begin{definition} \cite[II 2.3.2]{RaynaudG} \label{DSML}
An $R$-module $M$ is said to be a strict Mittag-Leffler $R$-module if for every 
finitely generated submodule $N\overset i\subseteq M$ there exist a finitely presented $R$-module $P$ and a commutative diagram of morphisms of $R$-modules
$$\xymatrix @R=8pt {& M \ar[dr]^-f & \\N \ar[ru]^-i \ar[rd]^-i & & P \ar[dl]^-g\\ & M  & }$$

\end{definition}

\begin{theorem} \label{7.8} $M$ is a strict Mittag-Leffler right $R$-module iff $\mathcal M$ is an SML module.\end{theorem}

\begin{proof} $\Rightarrow)$ For every finitely generated $R$-submodule
$N\overset{i_R}\subseteq M$ there exist a finitely presented (right) $R$-module $P$ and a commutative diagram of morphisms of $\RR$-modules
$$\xymatrix @R=8pt {& \mathcal M \ar[dr]^-f & \\ \mathcal N \ar[ru]^-i \ar[rd]^-i & &  \mathcal  P \ar[dl]^-g\\ &  \mathcal  M  & }$$
(where $i$ is the morphism induced by $i_R$). Hence, $\tilde N:=\Ima i\simeq \Ima(f\circ i)$, which is an FP module by \ref{4.6b}. The morphism $\mathcal M^*\to \tilde N^*$ is an epimorphism since the composite morphism $\mathcal P^*\to \mathcal M^*\to \tilde N^*$ is an epimorphism, by \ref{5.4}. By \ref{5.6}, $\mathcal M$ is an SML module.

$\Leftarrow)$ Let $N\subseteq M$ be  a finitely generated $R$-submodule.
The image  $\tilde N$ of  the induced morphism $\mathcal N\to \mathcal M$
is an FP module, by \ref{5.6}. Let $\tilde i\colon \tilde N\hookrightarrow \mathcal M$ be the inclusion morphism.
There exist a finitely presented (right) $R$-module $P$ and  a monomorphism $j\colon \tilde N\hookrightarrow \mathcal P$, by \ref{TH}. There exist a morphism
$f\colon \mathcal M\to \mathcal P$ such that $j=f\circ i$
since the morphism $\mathcal M^*\to \tilde N^*$ is an epimorphism (by \ref{ains}) and $j\in \tilde N^*(P)$. There exists a morphism $g\colon \mathcal P \to \mathcal M$ such that
$i=g\circ j$ since the morphism $\mathcal P^*\to \tilde N^*$ is an epimorphism by \ref{ains} and $i\in \tilde N^*(M)$.  The diagram
$$\xymatrix @R=8pt { & & \mathcal M \ar[dr]^-f & \\ \mathcal N \ar[r] & \tilde N \ar@{-->}[rr]^-j \ar[ru]^-{\tilde i} \ar[rd]_-{\tilde i} & &  \mathcal  P \ar[dl]^-g\\ & &  \mathcal  M  & }$$
is commutative. Hence, $M$ is a strict Mittag-Leffler right module.

\end{proof}

\section{Appendix: Functor $\mathbb D$}

Let $\mathbb D$ be the contravariant additive functor from the category of abelian groups to the category of abelian groups defined by
$$\mathbb D(N)=\Hom_{\ZZ}(N,\QQ/\ZZ),$$
for any abelian group $N$. $\mathbb D$ is an exact functor and $\mathbb D(N)=0$ iff $N=0$. 

We will say that $\mathbb H$ is a contravariant $\mathcal R$-module if $\mathbb H$ is a contravariant additive functor from the category of $R$-modules to the category abelian groups.

\begin{proposition} \label{Tlemma} Let $\mathbb M$ be an $\RR$-module and $\mathbb H$ a contravariant $\RR$-module. Then, there exists
a natural isomorphism 
$$\Hom_{\RR}(\mathbb M,\mathbb D\circ \mathbb H)=\Hom_{\RR}(\mathbb H,\mathbb D\circ \mathbb M).$$

In particular, if $\mathbb M$ is projective then $\mathbb D\circ \mathbb M$ is injective, and
if $\mathbb H$ is projective then $\mathbb D\circ \mathbb H$ is injective.

\end{proposition}

\begin{proof} $\Hom_{\ZZ}(\mathbb M(S),\Hom_{\ZZ}(\mathbb H(S),\QQ/\ZZ))  
 =\Hom_{\ZZ}(\mathbb H(S),\Hom_{\ZZ}(\mathbb M(S),\QQ/\ZZ)),$ for any $R$-module $S$.

\end{proof}

\begin{example} Given an $R$-module $N'$, let $N'_\bullet$  be the contravariant $\RR$-module defined by
$$N'_\bullet(S):=\Hom_R(S,N')$$
Observe that $N'_\bullet$ is a projective $\mathcal R$-module  since $\Hom_{\RR}(N'_\bullet,\mathbb H)=\mathbb H(N')$. 
Let $N$ be a right $R$-module. $\mathbb D\circ \mathcal N=\mathbb D(N)_\bullet$. Hence, 
$\mathbb D^2\circ \mathcal N=\mathbb D\circ \mathbb D(N)_\bullet$ and it is an injective (right)  $\RR$-module by Proposition \ref{Tlemma}. By Proposition \ref{AdjQs},
$(\mathbb D^2\circ \mathcal N)_{\langle\text{\sl Qs-ch}\rangle}$ is an injective 
object of $\langle\text{\sl Qs-ch}\rangle$. By Proposition \ref{pure}, $(\mathbb D^2\circ \mathcal N)_{\langle\text{\sl Qs-ch}\rangle}(R)=\mathbb D^2(N)$ is a pure-injective $R$-module.

\end{example}

Given a module $N$, let $i_N\colon N\to \mathbb D^2(N)$ be the natural morphism defined by $i_N(n)(w):=w(n)$, for any $n\in N$ and $w\in \mathbb D(N)=\Hom_{\ZZ}(N,\QQ/\ZZ)$. The  morphism $i_N$ is a monomorphism and  
it is easy to check that the composite morphism
$$\mathbb D(N) \overset{i_{\mathbb D(N)}}\longrightarrow \mathbb D^3(N) 
\overset{\mathbb D(i_N)} \longrightarrow \mathbb D(N) $$
is the identity morphism.

\begin{proposition} \label{Tlemma2} Let $\mathbb M$ be a right $\RR$-module and $\mathbb H$ a contravariant $\RR$-module. Then, there exists
a natural isomorphism 
$$\Hom_{\RR}(\mathbb M\circ \mathbb D, \mathbb H)=\Hom_{\RR}(\mathbb M,\mathbb H\circ \mathbb D).$$

In particular, if $\mathbb M$ is projective then $\mathbb M\circ \mathbb D$ is projective, and
if $\mathbb H$ is injective then $\mathbb H\circ \mathbb D$ is injective.
\end{proposition}

\begin{proof} Given a morphism $f\colon \mathbb M\circ \mathbb D\to  \mathbb H$ composing with $\mathbb D$ we obtain the morphism 
$\mathbb M\circ \mathbb D^2\to  \mathbb H\circ \mathbb D$. Let $f'$ be the composite morphism $\mathbb M\to \mathbb M\circ \mathbb D^2\to  \mathbb H\circ \mathbb D$, that is, $f'_N=f_{\mathbb D(N)} \circ \mathbb M(i_N)$, for any $R$-module $N$.

Given a morphism $g\colon \mathbb M\to  \mathbb H\circ \mathbb D$ composing with $\mathbb D$ we obtain the morphism 
$\mathbb M\circ \mathbb D\to  \mathbb H\circ \mathbb D^2$. Let $\widehat g$ be the composite morphism $\mathbb M\circ \mathbb D\to  \mathbb H\circ \mathbb D^2\to \mathbb H$, , that is, $\widehat g_N=\mathbb H(i_N)\circ g_{\mathbb D(N)} $, for any $R$-module $N$.

We have to prove that $f=\widehat{f'\,}$ and $g={\widehat g\,}'$.
The diagram 
$$\xymatrix{\mathbb M(\mathbb D(N)) \ar[r]^-{\mathbb M(i_{\mathbb  D(N)})}
\ar[rd]_-{Id} & \mathbb M(\mathbb D^3(N)) \ar[r]^-{f_{\mathbb  D^2(N)}} \ar[d]^-{\mathbb M(\mathbb D(i_{N}))} & \mathbb  H(\mathbb  D^2(N)) \ar[d]^-{\mathbb H(i_N)}\\ & \mathbb M(\mathbb D(N)) \ar[r]^-{f_N} & \mathbb H(N)}$$
is commutative. Hence, $f_N=\mathbb H(i_N)\circ f_{\mathbb  D^2(N)}\circ \mathbb M(i_{\mathbb  D(N)})=\mathbb H(i_N)\circ f'_{{\mathbb D(N)}}=\widehat{f'\,}_N$. 

Likewise, $g={\widehat g\,}'$.

\end{proof}

%

Let $N$ be a right $R$-module and let $c\colon N\otimes_R \mathbb D(N)\to \QQ/\ZZ$ be defined by $c(n\otimes w):=w(n)$. 
We have a natural morphism
$$\mathbb M^*\circ \mathbb D\to \mathbb D\circ \mathbb M, \, f\mapsto \tilde f \text{ for any } f\in \mathbb M^*(\mathbb D (N))$$
where $\tilde f\in (\mathbb D\circ \mathbb M)(N)=\Hom_{\ZZ}(\mathbb M(N),\QQ/\ZZ)$ is defined by $\tilde f(m):=c(f_{N}(m))$. By Proposition \ref{Tlemma2},
we have the natural morphism
$$\mathbb M^*\to \mathbb D\circ \mathbb M\circ \mathbb D, \, f\mapsto \tilde f,
\text{ for any } f\in \mathbb M^*(N)=\Hom_{\mathcal R}(\mathbb M,\mathcal N),$$
where $\tilde f\in \mathbb D(\mathbb M(\mathbb D(N)))=\Hom_{\ZZ}(\mathbb M(\mathbb D(N)),\QQ/\ZZ)$ is defined as $\tilde f(m)=c(f_{\mathbb D(N)}(m))$.

\begin{proposition} The natural morphism
$$\mathbb M^*\to \mathbb D\circ \mathbb M\circ \mathbb D$$
is a monomorphism.  In particular, $\mathbb M^*$ is a well defined functor, that is, $\mathbb M^*(N)$ is a set for any $R$-module $N$.

\end{proposition}

\begin{proof} The composite morphism
$$\aligned \mathbb M^*(N) & =\Hom_{\RR}(\mathbb M,\mathcal N)\hookrightarrow \Hom_{\RR}(\mathbb M,\mathbb D^2\circ \mathcal N)=\Hom_{\RR}(\mathbb D\circ \mathcal N,\mathbb D\circ \mathbb M)\\ &=\Hom_{\RR}(\mathbb D(N)_\bullet,\mathbb D\circ \mathbb M) =(\mathbb D\circ \mathbb M)(\mathbb D(N))=
(\mathbb D\circ \mathbb M\circ \mathbb D)(N)\endaligned$$
is injective.
\end{proof}

%
%
%
%

By Proposition \ref{Tlemma2}, we have the natural morphism
$$\aligned \mathbb M\circ \mathbb D\to \mathbb D\circ \mathbb M^*,\, m\mapsto \tilde m\, \text{ for any } m\in \mathbb M(\mathbb D(N))\\ 
\endaligned$$
where $\tilde m\in (\mathbb D\circ \mathbb M^*)(N)=\Hom_{\ZZ}(\Hom_{\mathcal R}(\mathbb M,\mathcal N),\QQ/\ZZ)$ is defined by $\tilde m(g):=c(g_{\mathbb D(N)}(m))$, and the triangle 
$$\xymatrix{ \mathbb M\circ \mathbb D \ar[rr] \ar[rd] & &  \mathbb D\circ \mathbb M^*\\ & \mathbb M^{**}\circ \mathbb D \ar[ur]  &}$$
is commutative.

\begin{proposition} \label{T-lemma} If $\mathbb M$ preserves direct limits, then the natural morphism $$\mathbb M^*\circ \mathbb D\to \mathbb D\circ \mathbb M$$ is an isomorphism.
\end{proposition}

\begin{proof} 
The contravariant functors $\mathbb M  \rightsquigarrow \mathbb M^*\circ \mathbb D,\, \mathbb D\circ \mathbb M$ are  left-exact   and transform direct sums into direct products. By Theorem \ref{4.6N}, we have only to check that
$\mathcal P\circ \mathbb D=\mathbb D\circ \mathcal P^*$, for any finitely presented $R$-module.
The functors $M  \rightsquigarrow \mathcal M\circ \mathbb D,\, \mathbb D\circ \mathcal  M^*$ are right exact and  transform finite direct sums into finite direct sums and $\mathcal R\circ \mathbb D=\mathbb D\circ \mathcal R^*$, and we conclude.


%
%
%
%
%
\end{proof}

%
%


\begin{thebibliography}{99}

\bibitem{Amel} \textsc{\'Alvarez, A., Sancho, C., Sancho, P.,}
\textit{\!Algebra schemes and their representations}, J. Algebra
{\bf 296/1} (2006) 110-144.




\bibitem{Drinfeld} \textsc{Drinfeld, V.,} \textit{\! Infinite dimensional vector bundles in algebraic geometry: an introduction, in ``The Unity of Mathematics''}, Birkha\"{u}ser, Boston 2006, 263-304.



\bibitem{Grot} \textsc{Grothendieck, A.} \textit{\! EGA, III.}  Math. Inst. Hautes Etudes Scient. 11 (1961)


\bibitem{groty}  \textsc{Grothendieck, A.}   \textit{\! Sur quelques points d'alg\`ebre homologique}, Tohoku Mathematical Journal, (2), {\bf 9}: 119-221, (1957).

\bibitem{Guil} \textsc{Guil, P.A.,  Izurdiaga, M.C.,
Rothmaler, P., Torrecillas, B.,}
\textit{\! Strict Mittag-Leffler modules}.
MLQ Math. Log. Q. 57 (2011), no. 6, 566-570.


\bibitem{Herbera} \textsc{Herbera, D.; Trlifaj, J.,} \textit{\! Almost free modules and Mittag-Leffler conditions}, Avd. Math. 229 (6),  (2012),  3436-3467

\bibitem{RaynaudG} \textsc{Gruson,L., Raynaud, M.,} \textit{\!Crit\`eres de platitude et de projectivit\'{e}}, Inventiones math. 13, 1-89 (1971).

\bibitem{Matsumura} \textsc{Matsumura, H.,} \textit{\! Commutative ring theory}, Cambridge University Press (2000).


\bibitem{Pedro4} \textsc{Gordillo A., Navarro J.,Sancho, P.,} \textit{\! Functors of modules associated with flat and projective modules}  2018 arXiv:1710.04153v3.


\bibitem{stacks-project}  \textit{\! Stacks Project}  Version e7f99af, compiled on Oct 25, 2016.

\bibitem{Watts} \textsc{Watts C.E.,} \textit{\! Intrinsic characterizations of some additive functors}, PAMS (1959)

\end{thebibliography}
\end{document}